\documentclass[12pt]{article}
\usepackage{amsthm,amssymb,amsmath,mathrsfs}
\usepackage{epsfig}
\usepackage[usenames]{color}
\usepackage{graphicx}
\newtheorem{prelem}{{\bf Theorem}}

 \newtheorem{theorem}{Theorem}

\newtheorem{lemma}[theorem]{Lemma}

\theoremstyle{definition}

\theoremstyle{remark}

\theoremstyle{conjecture}

\title{A new search direction for full-Newton step infeasible interior-point method in linear optimization}
\date{}
\date{}
\author
{B. Kheirfam \\
Department of Applied Mathematics\\ Azarbaijan Shahid Madani
University\\
Tabriz,  Iran\\
 \tt b.kheirfam@azaruniv.ac.ir}

\begin{document}

\maketitle
\begin{abstract}
In this paper, we study an infeasible interior-point method for
linear optimization with full-Newton step. The introduced method
uses an algebraic equivalent transformation on the centering
equation of the system which defines the central path. We prove that
the method finds an $\varepsilon$-optimal solution of the underlying
problem in polynomial time.

\end{abstract}

\noindent{\bf keywords}: Linear optimization, infeasible
interior-point methods, new search directions, polynomial
complexity.\\

\noindent{\bf AMS}: 90C05,  90C51.


\section{Introduction}
Interior-point methods (IPMs) for linear optimization (LO) began
when Karmarkar \cite{karmarkar} published his exceptional paper in
1984. After that, several variants of this algorithm were presented.
In the meantime, we can talk about feasible and infeasible IPMs. In
feasible IPMs we presume that a strictly feasible point is at hand
which the algorithm can be immediately  beginning. Usually find such
a starting point is not simple. In that case an infeasible IPM
(IIPM) should be used. These methods begin from an arbitrary
positive point and try to reach both feasibility and optimality.
IIPMs were first introduced by Lustig \cite{Lustig} and Tanabe
\cite{Tanabe}. The first feasible IPM with full-Newton step for LO
was presented by Roos et al. \cite{Terlaky}. Determining the search
directions plays a very important role in IPMs. In 2003, Darvay
\cite{Darvay} utilizes the AET technique on the centering equation of
the system defining the central path for LO. He uses the square root
function in the AET strategy and then applies the Newton method for obtain
the search directions. This method is extended in
\cite{Achache,WBS,WangBai,Letters}, respectively, to convex
quadratic optimization (CQO), second-order cone optimization (SOCO),
symmetric optimization (SO) and the Cartesian $P_*(\kappa)$ linear
complementarity problem (LCP). Kheirfam and Haghighi \cite{Haghighi}
have proposed an IPM for $P_*(\kappa)$-LCP which uses the function
$\psi(t)=\sqrt{t}/2(1+\sqrt{t})$ in the AET technique. An infeasible
version of the method proposed in \cite{Terlaky} has presented by
Roos in \cite{Roos} which needs a feasibility step and three
centering steps in each main iteration. Some generalizations and
versions of the method can be seen in
\cite{Bulletin,Kheirfamjota,NACO,Liu,Zhang,KheirfamAnnals}. The
author is improved this algorithm so that the algorithm performs
only one feasibility step in each iteration and does not need
centering steps \cite{NewRoos}. Kheirfam \cite{Num,Let,Asian,Num1}
extended the algorithm proposed in \cite{NewRoos} to HLCP, the
Cartesian $P_*(\kappa)$-LCP, the convex quadratic symmetric cone
optimization (CQSCO) and SO. By considering the AET technique based
on the function $\psi(t)=t-\sqrt{t}$, Darvay et al.
\cite{Darvay2016} have introduced a full-Newton step IPM for LO.
Kheirfam \cite{SDLCP} has presented an infeasible version of this
algorithm for SDLCP. Darvay et al. \cite{CPDarvay} published a
corrector-predictor IPM (CP-IPM) for LO using the function
$\psi(t)=t-\sqrt{t}$ for AET. Darvay and Tak\'acs \cite{Takacs}
proposed an IPM for LO based on a new type of AET on the centering
equation of the central path.

Motivated by the aforementioned works, in this paper we aim to
present a full-Newton step IIPM for LO using the AET
$\psi(t)=\psi(\sqrt{t})$ for the centering equation of the central
path. The method uses the function $\psi(t)=t^2$ in order to
determine the new search directions and performs only one
feasibility step in a main iteration. In fact, our method is an
infeasible version of the method proposed in \cite{Takacs}. We prove
that the proposed algorithm enjoys the best-known iteration
complexity for IIPMs.

The paper is organized in the following way. In the next section, we
remember the problem pair (P) and (D). We state the perturbed
problems corresponding to (P) and (D) and then provided the central
path. In Sect. 3, the new search directions based on the new type of
AET using the function $\psi(t)=t^2$ is discussed, and finally the
algorithm is presented. Section 4 consists of the complexity
analysis of the introduced IIPM with the new search directions.
 In Section 5, some concluding remarks are followed.

\section{Preliminaries}\label{sec2}
Let us consider the LO problem in the standard form
\begin{eqnarray*}
(P)~~~~~~~~~\min~\{c^Tx:~Ax=b, ~x\geq0\},~~~~~
\end{eqnarray*}
where $A\in{R}^{m\times n}$ with $rank(A)=m, b\in R^m$ and $c\in
R^n$. The dual of this problem can be written in the following standard
form:
\begin{eqnarray*}
(D)~~~~~~~~~\max~\{b^Ty:~A^Ty+s=c, ~ s\geq0\}.
\end{eqnarray*}
In accordance with the routine of IIPMs, we consider the starting
point $(x^0, y^0, s^0)=\xi(e, 0, e)$ such that $\|(x^*;
s^*)\|_{\infty}\leq\xi$ for  some primal-dual optimal solution
$(x^*, y^*, s^*)$, where $e$ is the all-one vector and $\xi$ is a
positive scalar. It should be noted that for the optimal solution
$(x^*, y^*, s^*)$ the inequality $\|(x^*; s^*)\|_{\infty}\leq\xi$ is
true if and only if
\begin{eqnarray}\label{eq1}
0\leq x^*\leq\xi e,~~0\leq s^*\leq \xi e.
\end{eqnarray}
For an IIPM, a triple $(x, y, s)$ is called an $\varepsilon$-solution
of (P) and (D) if
$$\max\big\{x^Ts, \|b-Ax\|, \|c-A^Ty-s\|\big\}\leq\epsilon,$$
where $\varepsilon$ is a accuracy parameter. Following \cite{Roos}, for
any $0<\nu\leq1$ we consider the perturbed problem pair
${\rm(P_{\nu})}$ and ${\rm (D_{\nu})}$ as follows:
\begin{eqnarray*}
(P_{\nu})~~~~~~~~~\min~\{(c-\nu r^0_c)^Tx:~b-Ax=\nu r_b^0,
~x\geq0\},~~~~~
\end{eqnarray*}
\begin{eqnarray*}
(D_{\nu})~~~~~~~~~\max~\{(b-\nu r^0_b)^Ty:~c-A^Ty-s=\nu r^0_c, ~
s\geq0\},
\end{eqnarray*}
where $r^0_b:=b-A\xi e$ and $r^0_c:=c-\xi e.$ It is simply seen that
$(x^0, y^0, s^0)=\xi(e, 0, e)$ is a feasible solution of the problem
pair ${\rm(P_{\nu})}$ and ${\rm (D_{\nu})}$ if $\nu=1$. We conclude
that if $\nu=1$, then ${\rm(P_{\nu})}$ and ${\rm (D_{\nu})}$ satisfy
the interior point condition (IPC). We recall the following lemma.
\begin{lemma}{\rm(Theorem 5.13 in \cite{Yebook})}\label{lem1}
The original problems, {\rm(P)} and {\rm(D)} are feasible if and
only if for each $\nu$ satisfying $0<\nu\leq1$ the perturbed
problems ${\rm(P_{\nu})}$ and ${\rm(D_{\nu})}$ satisfy the IPC.
\end{lemma}
In the view of Lemma \ref{lem1}, we assume that the original problem
pair (P) and (D) is feasible and $\nu\in(0, 1]$, the central path of
the perturbed pair ${\rm(P_{\nu})}$ and ${\rm(D_{\nu})}$ exists;
\begin{eqnarray}\label{eq2}
\begin{array}{ccccccc}
b-Ax=\nu r^0_b,~~ x\geq0,\\
c-A^Ty-s=\nu r^0_c,~~ s\geq0,~~~~~~~\\
xs=\mu e,~~~~~~
\end{array}
\end{eqnarray}
has a unique solution $(x(\mu, \nu), y(\mu, \nu), s(\mu, \nu))$ for
every $\mu>0$. This solution consists of the $\mu$-centers of the
perturbed problems ${\rm(P_{\nu})}$ and ${\rm(D_{\nu})}$. Note that
for $x, s>0$ and $\mu>0$ from the third equation of system
(\ref{eq2}) we deduce that
\begin{eqnarray}\label{eq2a}
xs=\mu e\Leftrightarrow \frac{xs}{\mu}=e \Leftrightarrow
\sqrt{\frac{xs}{\mu}}=e \Leftrightarrow
\frac{xs}{\mu}=\sqrt{\frac{xs}{\mu}}.
\end{eqnarray}
Now the perturbed central path can be equivalently stated as follows:
\begin{eqnarray}\label{eq2b}
\begin{array}{ccccccc}
b-Ax=\nu r^0_b,~~ x\geq0,\\
c-A^Ty-s=\nu r^0_c,~~ s\geq0,~~~~~~~\\
\dfrac{xs}{\mu}=\sqrt{\dfrac{xs}{\mu}}.~~~
\end{array}
\end{eqnarray}
In the sequel, the parameters $\mu$ and $\nu$ always satisfy the
relation $\mu=\nu\mu^0=\nu\xi^2$.
\section{New search directions}
In accordance with the Darvay's idea, we consider the function
$\psi$ defined and continuously differentiable on the interval
$(k^2, \infty)$, where $0\leq k<1$, such that
$2t\psi^{'}(t^2)-\psi^{'}(t)>0, \forall t>k^2$. Now, if we apply the
AET method to (\ref{eq2b}), then we get
\begin{eqnarray}
b-Ax=\nu r^0_b,~~ x\geq0,\label{eq3}\\
c-A^Ty-s=\nu r^0_c,~~ s\geq0, \label{eq4}\\
~~~~~~~~~~\psi\Big(\dfrac{xs}{\mu}\Big)=\psi\Big(\sqrt{\dfrac{xs}{\mu}}\Big).~~~\label{eq5}
\end{eqnarray}
Let $(x, y, s)$ be a feasible solution of the perturbed pair
${\rm(P_{\nu})}$ and ${\rm (D_{\nu})}$. We consider the notation
$$f(x, y, s)=\begin{bmatrix}
               \nu^+ r^0_b-b-Ax \\
               \nu^+ r^0_c-c+A^Ty+s \\
               \psi\Big(\dfrac{xs}{\mu}\Big)-\psi\Big(\sqrt{\dfrac{xs}{\mu}}\Big)\\
             \end{bmatrix}=0,$$
where $\nu^+=(1-\theta)\nu$ and $\theta\in(0, 1)$. Applying Newton's
method to this system, we get
$$J_f(x, y, s)\begin{bmatrix}
                \Delta x \\
                \Delta y \\
                \Delta s \\
              \end{bmatrix}=-f(x, y, s),$$
              where $J_f(x, y, s)$ denotes the Jacobian matrix of
              $f$ at $(x, y, s)$. After some computations, we obtain
              the following system:
\begin{eqnarray}\label{eq6}
\begin{array}{ccccccc}
A\Delta x=\theta\nu r^0_b,~~~~~~~~~~\\
A^T\Delta y+\Delta s=\theta\nu r^0_c,~~~~~~~~~~~~~~~~~~~\\
\frac{1}{\mu}\big(s\Delta x+x\Delta
s\big)=\dfrac{-\psi(\frac{xs}{\mu})+\psi(\sqrt{\frac{xs}{\mu}})}{\psi^{'}(\frac{xs}{\mu})-
\frac{1}{2\sqrt{\frac{xs}{\mu}}}\psi^{'}(\sqrt{\frac{xs}{\mu}})}.
\end{array}
\end{eqnarray}
Defining the scaled search directions
\begin{eqnarray}\label{eq7}
 d_x:=\frac{v\Delta x}{x},~~ d_s:=\frac{v\Delta s}{s},~~{\rm
 where}~v=\sqrt{\frac{xs}{\mu}},
\end{eqnarray}
we can  give the scaled form of system (\ref{eq6}):
\begin{eqnarray}\label{eq8}
\begin{array}{ccccccc}
{\bar A}d_x=\theta\nu r^0_b,~~~~~~\\
{\bar A}^T\frac{\Delta y}{\mu}+d_s=\theta\nu vs^{-1}r^0_c,~~~~~~~~~\\
d_x+d_s=p_{v},~~~~~~~~~~~~~
\end{array}
\end{eqnarray}
where
$$p_v:=\frac{2\psi(v)-2\psi(v^2)}{2v\psi^{'}(v^2)-\psi^{'}(v)},~~{\rm and}~~ \bar A:=A{\rm diag}\big(\frac{x}{v}\big).$$
If we use the function $\psi:(\frac{1}{\sqrt2}, \infty)\rightarrow
\mathbb{R}, \psi(t)=t^2$ introduced in \cite{Takacs}, then we obtain
\begin{eqnarray}\label{eq9}
p_v=\frac{v-v^3}{2v^2-e}.
\end{eqnarray}
After a full-Newton step, the new iterate is given by
\begin{eqnarray}\label{eq10}
x_+:=x+\Delta x, ~~~ y_+:=y+\Delta y,~~~ s_+:=s+\Delta s.
\end{eqnarray}
Furthermore, in each iteration of the algorithm, a quantity is
needed to measure how far an iterate is from the central path. We
consider the proximity measure defined by
\begin{eqnarray}\label{eq11}
\delta(v):=\delta(x, s;
\mu)=\frac{\|p_v\|}{2}=\frac{1}{2}\Big\|\frac{v-v^3}{2v^2-e}\Big\|,
\end{eqnarray}
which was first suggested for a feasible IPM in \cite{Takacs}.

Let $q_v=d_x-d_s$. Then
\begin{eqnarray}\label{eq12}
d_x=\frac{p_v+q_v}{2}, ~~~ d_s=\frac{p_v-q_v}{2},~~~
d_xd_s=\frac{p_v^2-q_v^2}{4},
\end{eqnarray}
and
\begin{eqnarray}\label{eq13}
\frac{\|q_v\|^2}{4}=\frac{\|d_x-d_s\|^2}{4}=\frac{\|d_x+d_s\|^2}{4}-d_x^Td_s=\frac{\|p_v\|^2}{4}-d_x^Td_s.
\end{eqnarray}

Suppose that for some $\mu\in(0, \mu^0]$, our algorithm begins from
a feasible solution $(x, y, s)$ of the problem pair ${\rm(P_{\nu})}$
and ${\rm (D_{\nu})}$ with $\nu=\frac{\mu}{\mu^0}$, and such that
$\delta(x, s; \mu)\leq\tau, \tau\in(0, 1)$. Then, the algorithm
finds a feasible solution $(x_+, y_+, s_+)$ of ${\rm(P_{\nu}^+)}$
and ${\rm (D_{\nu}^+)}$, where $\nu^+=(1-\theta)\nu, \theta\in(0,
1)$. In this case, $\mu$ is decreased to $\mu^+=(1-\theta)\mu$ and
such that $\delta(x_+, s_+; \mu^+)\leq\tau$. This procedure is
repeated until an $\varepsilon$-solution is found. We are now in a
position to state the theoretical framework of the infeasible
interior-point algorithm as follows:
$$\begin{array}{cccccc}\hline
 {\rm {\bf Algorithm 1}:an~infeasible~interior-point~algorithm} \\
\hline~~~~~~ {\bf
Input:}~~~~~~~~~~~~~~~~~~~~~~~~~~~~~~~~~~~~~~~~~~~~~~~~~~~~~~~~~~~~~~~~~~~~~~~~~~~~~~~~~~~~~~\\~
{\rm~Accuracy~ parameter~ \varepsilon >0};~~~~~~~~~~~~~~~~~~~~~~~~~~~~~~~~~~~~~~\\
{\rm~barrier~ update~ parameter~ \theta,~0<\theta<1;}~~~~~~~~~~~~~~~~~~~~~\\
{\rm~threshold~parameter~}\tau>0.~~~~~~~~~~~~~~~~~~~~~~~~~~~~~~~~~~~~~\\
{\bf begin}~~~~~~~~~~~~~~~~~~~~~~~~~~~~~~~~~~~~~~~~~~~~~~~~~~~~~~~~~~~~~~~~~~~~~~~~~~~~~~~~~\\
x:=\xi e;~y:=0;~s:=\xi e;~\mu:=\nu\xi^2;~\nu=1;~~~~~~~~~~~~~~~~~~~~~~\\
{\bf while}~ \max(x^Ts, \|r_b\|, \|r_c\|)>\varepsilon~ {\bf do}~~~~~~~~~~~~~~~~~~~~~~~~~~~~~~~~~~~~~~~~\\
{\bf begin}~~~~~~~~~~~~~~~~~~~~~~~~~~~~~~~~~~~~~~~~~~~~~~~~~~~~~~~~~~~~~~~~~~~~~~~~\\
{\rm~solve~ the ~system~(\ref{eq8})~and ~use~(\ref{eq7})~to~obtain}~(\Delta x, \Delta y, \Delta s);\\
(x, y, s):=(x, y, s)+(\Delta x, \Delta y, \Delta s);~~~~~~~~~~~~~~~~~~~~~\\
{\rm~update~of~}\mu ~{\rm and}~\nu:~~~~~~~~~~~~~~~~~~~~~~~~~~~~~~~~~~~~~~~~~~~~~\\
 \mu:=(1-\theta)\mu;~~~~~~~~~~~~~~~~~~~~~~~~~~~~~~~~~~~~~\\
\nu:=(1-\theta)\nu;~~~~~~~~~~~~~~~~~~~~~~~~~~~~~~~~~~~~~\\
{\bf end}~~~~~~~~~~~~~~~~~~~~~~~~~~~~~~~~~~~~~~~~~~~~~~~~~~~~~~~~~~~~~~~~~~~~~~~~~\\
{\bf end}.~~~~~~~~~~~~~~~~~~~~~~~~~~~~~~~~~~~~~~~~~~~~~~~~~~~~~~~~~~~~~~~~~~~~~~~~~~~~~~~~~~\\
\hline
\end{array}$$

\section{Analysis of the algorithm}
Here, we will prove that Algorithm 1 is well-defined. The main goal
of our analysis is to find some values for the parameters $\tau$ and
$\theta$ such that $x_+>0$ and $s_+>0$, and  we have $\delta(x_+,
s_+; \mu^+)\leq\tau.$  In the following  section, we obtain an upper
bound for the proximity measure after an iteration of the algorithm.

\subsection{Upper bound for $\delta(v^+)$}
In the next lemma, we give a condition on the proximity measure which
ensures the feasibility of a full-Newton step. In what follows, we
use the notation $\omega=\frac{1}{2}\big(\|d_x\|^2+\|d_s\|^2\big).$

\begin{lemma}\label{lem3} The iterate $(x_+, y_+, s_+)$ with
$v>\frac{1}{\sqrt2}e$ is strictly feasible if
$\delta(v)^2+\omega<1$.
\end{lemma}
\begin{proof}
Let $0\leq\alpha\leq1$. We define $x(\alpha):=x+\alpha\Delta x$ and
$s(\alpha):=s+\alpha\Delta s.$ Using (\ref{eq7}), the third equation
of (\ref{eq8}) and (\ref{eq12}) one can find
\begin{eqnarray}\label{eq14}
\frac{x(\alpha)s(\alpha)}{\mu}=\frac{xs}{v^2}(v+\alpha d_x)(v+\alpha
d_s)=v^2+\alpha v(d_x+d_s)+\alpha^2 d_xd_s \nonumber\\
=(1-\alpha)v^2+\alpha(v^2+v p_v)+\alpha^2\Big(\frac{p_v^2-q_v^2}{4}\Big)~~~~~~~~~~~~~~~\\
\geq(1-\alpha)v^2+\alpha^2
e+\alpha^2\frac{p_v^2}{4}-\alpha^2\frac{q_v^2}{4},~~~~~~~~~~~~~~~~~~~~~~~~\nonumber
\end{eqnarray}
where the inequality is due to $\alpha\geq\alpha^2$ and the
following inequality:
\begin{eqnarray}\label{eq15}
v^2+v
p_v-e=v^2+\frac{v^2-v^4}{2v^2-e}-e=\frac{v^4}{2v^2-e}-e=\frac{(v^2-e)^2}{2v^2-e}\geq0.
\end{eqnarray}
The inequality $x(\alpha)s(\alpha)>0$ holds if
\begin{eqnarray*}
\Big\|-\frac{p_v^2}{4}+\frac{q_v^2}{4}\Big\|_{\infty}
\leq\Big\|\frac{p_v^2}{4}\Big\|_{\infty}+\Big\|\frac{q_v^2}{4}\Big\|_{\infty}
\leq\frac{\|p_v\|^2}{4}+\frac{\|q_v\|^2}{4}~~~~~~~~~~~~~~~~~~~~\nonumber\\
={\delta(v)^2}-{d_x^Td_s}\leq{\delta(v)^2}+\|d_x\|\|d_s\|\leq\delta(v)^2+\omega<1,
\end{eqnarray*}
where the equality is due to (\ref{eq13}), the third inequality uses
from the Cauchy-Schwarz inequality and the last inequality holds due
to the assumption of the lemma. Thus, $x(\alpha)s(\alpha)>0$, for
$0\leq\alpha\leq1;$  $x(\alpha)$ and $s(\alpha)$ do not change sign
on the interval $[0, 1]$. Consequently, $x(0)=x>0$ and $s(0)=s>0$
yields $x(1)=x_+>0$ and $s(1)=s_+>0$. Thus, the proof is completed.
\end{proof}
In correspondence to the definition (\ref{eq11}), we have
$$\delta(v_+)=\delta(x_+, s_+; \mu^+)=\frac{1}{2}\Big\|\frac{v_+-v_+^3}{2v^2_+-e}\Big\|,~ {\rm
where}~ v_+=\sqrt{\frac{x_+s_+}{\mu^+}}.$$
\begin{lemma}\label{lem4}
Let  $\delta(v)^2+\omega<\frac{1}{2}(1-\theta)$ and
$v>\frac{1}{\sqrt2}e$. Then, $v_+>\frac{1}{\sqrt2}e$ and
$$\delta(v_+)\leq\frac{\sqrt{1-\delta(v)^2-\omega}\big(\theta\sqrt{n}+10\delta(v)^2+\omega\big)}
{2\sqrt{1-\theta}(2(1-\delta(v)^2-\omega)-(1-\theta))}.$$
\end{lemma}
\begin{proof}
Let $\alpha=1$. Then from (\ref{eq14}) it follows that
\begin{eqnarray*}
v_+^2=\frac{x_+s_+}{\mu^+}=\frac{v^2+vp_v+\frac{p_v^2}{4}-\frac{q_v^2}{4}}{1-\theta}=
\frac{e+\frac{(v^2-e)^2}{2v^2-e}+\frac{p_v^2}{4}-\frac{q_v^2}{4}}{1-\theta}\\
=\frac{e+\big(\frac{9v^2-4e}{v^2}\big)\frac{p_v^2}{4}-\frac{q_v^2}{4}}{1-\theta}
\geq\frac{e-\frac{q_v^2}{4}}{1-\theta},
\end{eqnarray*}
where the second equality is due to (\ref{eq15}) and the inequality
follows from the fact that $9v^2-4e\geq0.5e>0$. Consequently, we
have
\begin{eqnarray}\label{eq17}
\min(v_+)\geq\sqrt{\frac{1-\frac{1}{4}\|q_v\|^2_{\infty}}{1-\theta}}\geq\sqrt{\frac{1-\frac{1}{4}\|q_v\|^2}{1-\theta}}
\geq\sqrt{\frac{1-\delta(v)^2-\omega}{1-\theta}},
\end{eqnarray}
where the last inequality follows from (\ref{eq13}) and the
Cauchy-Schwarz inequality.

From $\delta(v)^2+\omega<\frac{1}{2}(1-\theta)$ it follows that
$\min(v_+)>\frac{1}{\sqrt2}$, hence $v_+>\frac{1}{\sqrt2}e$. Now, we
have
\begin{eqnarray}\label{eq18}
\delta(v_+)=\frac{1}{2}\Big\|\frac{v_+-v_+^3}{2v^2_+-e}\Big\|=\frac{1}{2}\Big\|\frac{v_+}{2v^2_+-e}\big(e-v_+^2\big)\Big\| ~~~~~~~~\nonumber\\
\leq\frac{\min(v_+)}{2(2\min(v_+)^2-1)}\big\|e-v_+^2\big\|~~~~~~~~~~~~~~~~~~~~\nonumber\\
\leq\frac{\sqrt{(1-\theta)(1-\delta(v)^2-\omega)}}{2(2(1-\delta(v)^2-\omega)-(1-\theta))}\big\|e-v_+^2\big\|.~~~~~
\end{eqnarray}
On the other hand, one has
\begin{eqnarray*}
\big\|e-v_+^2\big\|=\Big\|\frac{e+\big(\frac{9v^2-4e}{v^2}\big)\frac{p_v^2}{4}-\frac{q_v^2}{4}}{1-\theta}-e\Big\|~~~~~~~~~~~~~~~~~\\
\leq\frac{1}{1-\theta}\Big(\|\theta
e\|+\Big\|\Big(\frac{9v^2-4e}{v^2}\Big)\frac{p_v^2}{4}-\frac{q_v^2}{4}\Big\|\Big)\\
\leq\frac{1}{1-\theta}\Big(\theta\sqrt{n}+9\frac{\|p_v\|^2}{4}+\frac{\|q_v\|^2}{4}\Big)~~~~~~~~~\\
=\frac{1}{1-\theta}\big(\theta\sqrt{n}+10\delta(v)^2+\omega\big).~~~~~~~~~~~~~~
\end{eqnarray*}
Substituting this bound into (\ref{eq18}) gives us exactly the
desired result. Thus, the proof is completed.
\end{proof}
\subsection{Upper bound for $\omega$}
Following \cite{NewRoos}, let $\mathcal{N}:=\{\zeta: {\bar
A}\zeta=0\}$ denote the null space of the matrix $\bar A$. Then, the
 $\{\zeta: {\bar A}\zeta=\theta\nu r^0_b\}$ affine space equals
$\mathcal{N}+d_x$. Since the row space of $\bar A$ is the orthogonal
complement ${\mathcal{N}}^\perp$ of $\mathcal{N}$, thus
$d_s\in\theta\nu vs^{-1}r^0_c+{\mathcal{N}}^\perp.$ Also note that
$\mathcal{N}\cap{\mathcal{N}}^\perp=\{0\}$, and the affine spaces
$\mathcal{N}+d_x$ and ${\mathcal{N}}^\perp+d_s$ meet in a unique
point $q$. Applying a similar argument to Lemma 3.4 in
\cite{NewRoos}, we can find
\begin{eqnarray}\label{eq19}
2\omega\leq\|q\|^2+\bigg(\|q\|+\Big\|\frac{v-v^3}{2v^2-e}\Big\|\bigg)^2=\|q\|^2+\big(\|q\|+2\delta(v)\big)^2.
\end{eqnarray}
Again from \cite{NewRoos}, we have
\begin{eqnarray}\label{eq20}
\|q\|\leq\frac{\theta\big(n+\|v\|^2\big)}{\min(v)}.
\end{eqnarray}
By definition (\ref{eq11}) , we have
\begin{eqnarray*}
2\delta(v)=\Big\|\frac{v-v^3}{2v^2-e}\Big\|=\Big\|\frac{v^2+v}{2v^2-e}(e-v)\Big\|
\geq\frac{1}{2}\big\|e-v\big\|\geq\frac{1}{2}\big(\|v\|-\|e\|\big),
\end{eqnarray*}
which implies
$$\|v\|\leq\|e\|+4\delta(v)=\sqrt{n}+4\delta(v).$$
Furthermore, we have
$$4\delta(v)\geq\|e-v\|\geq|1-v_i|, i=1, \ldots, n.$$
This gives $\min(v)\geq 1-4\delta(v)$. Combining these two
inequalities with (\ref{eq20}), we will get
\begin{eqnarray}\label{eq21}
\|q\|\leq\frac{\theta\Big(n+\big(\sqrt{n}+4\delta(v)\big)^2\Big)}{1-4\delta(v)}.
\end{eqnarray}


\subsection{Values for $\theta$ and $\tau$}
In this section, we require finding values $\theta$ and $\tau$ such
that if $\delta(v)\leq\tau$ holds, then $\delta(v_+)\leq\tau$.  From
Lemma \ref{lem4}, it suffices to have
\begin{eqnarray}\label{eq22}
\frac{\sqrt{1-\delta(v)^2-\omega}\big(\theta\sqrt{n}+10\delta(v)^2+\omega\big)}
{2\sqrt{1-\theta}(2(1-\delta(v)^2-\omega)-(1-\theta))}\leq\tau,
\end{eqnarray}
provided that $\delta(v)^2+\omega<\frac{1}{2}(1-\theta)$. One can
easily see the right-hand-side of (\ref{eq21}) is monotonically
increasing with respect to $\delta(v)<1$. Hence, invoking
$\delta(v)\leq\tau$, we have
\begin{eqnarray*}
\|q\|\leq\frac{\theta\Big(n+\big(\sqrt{n}+4\tau\big)^2\Big)}{1-4\tau}.
\end{eqnarray*}
By substituting the above result into (\ref{eq19}) and using again
$\delta(v)\leq\tau$, we obtain
$$\omega\leq\frac{1}{2}\bigg[\bigg(\frac{\theta\big(n+\big(\sqrt{n}+4\tau\big)^2\big)}{1-4\tau}\bigg)^2
+\bigg(\frac{\theta\big(n+\big(\sqrt{n}+4\tau\big)^2\big)}{1-4\tau}+2\tau\bigg)^2\bigg]=:f(\tau).$$
We claim that
\begin{eqnarray}\label{eq23}
\chi(t):=\frac{\sqrt{1-t}}{2(1-t)-(1-\theta)}, ~~ 0\leq
t\leq\frac{1}{2}(1-\theta),
\end{eqnarray}
is increasing. Hence, $0\leq\delta(v)^2+\omega\leq\tau^2+f(\tau)$
implies $\chi(\delta(v)^2+\omega)\leq\chi(\tau^2+f(\tau))$.
Therefore, $\delta(v)^2+\omega\leq\frac{1}{2}(1-\theta)$ and
(\ref{eq22}) will certainly hold if
$$\tau^2+f(\tau)\leq\frac{1}{2}(1-\theta),~~~y(\tau):=\frac{\chi(\tau^2+f(\tau))(\theta\sqrt{n}+10\tau^2+f(\tau))}{2\sqrt{1-\theta}}\leq\tau.$$
If we take $\tau=\frac{1}{12}$ and $\theta=\frac{1}{22n}, n\geq4$,
then $\tau^2+f(\tau)=0.0645<0.4773\leq\frac{1}{2}(1-\theta)$ and
$y(\tau)\leq0.0827<\frac{1}{12}$. Hence, we may state the following
result.
\begin{lemma}\label{lem7}
If $\tau=\frac{1}{12}$ and $\theta=\frac{1}{22n}, n\geq4$, then
$\delta(v)\leq\tau$ implies $\delta(v_+)\leq\tau.$
\end{lemma}
\subsection{Complexity analysis}
Lemma \ref{lem7} establishes the proposed algorithm is well-defined,
in the sense that the property  $\delta(x, s;
\mu):=\delta(v)\leq\tau$ is maintained in all iterations.

In each main iteration, both the barrier parameter $\mu$ and the
norms of the residual vectors are reduced by the factor $1-\theta$.
Hence, the total number of main iterations is bounded above by
$$\frac{1}{\theta}\log\frac{\max\{n\xi^2, \|r_b^0\|, \|r_c^0\|\}}{\varepsilon}.$$
Now, we state our main result.
\begin{theorem}\label{thr}
If {\rm(P)} and {\rm(D)} are feasible and $\xi>0$ such that $\|(x^*;
s^*)\|_{\infty}\leq\xi$ for some optimal solutions $x^*$ of {\rm(P)}
and $(y^*, s^*)$ of {\rm(D)}, then after at most
$$22 n\log\frac{\max\{n\xi^2, \|r_b^0\|,
\|r_c^0\|\}}{\epsilon}$$ iterations, the algorithm finds an
$\epsilon$-optimal solution of {\rm(P)} and {\rm(D)}.
\end{theorem}
\section{Conclusions}
The method presented in this paper is a full-Newton step IIPM for LO
based on the AET proposed in \cite{Takacs}. The method is used in
each iteration only one feasibility step. Our method analysis is
different from the existing IIPMs based on the AET because it uses a
different AET. The obtained complexity bound coincides with the
current best-known theoretical iteration bound for IIPMs.


\end{document}